\documentclass[12pt]{amsart}

\usepackage{amssymb,amsmath,graphicx,color,textcomp, amsthm,bbm,bbold, enumerate,booktabs}
\usepackage{chngcntr}
\usepackage{apptools}
\AtAppendix{\counterwithin{theorem}{section}}
\definecolor{Red}{cmyk}{0,1,1,0}

\definecolor{verde}{cmyk}{1,0,1,0}

\definecolor{loka}{cmyk}{.5,0,1,.5}
\definecolor{azul}{cmyk}{1,1,0,0}

\numberwithin{equation}{section}

\newcommand{\be}{\begin{equation}}
\newcommand{\ee}{\end{equation}}

\newtheorem{theorem}{Theorem}

\newtheorem{definition}{Definition}
\newtheorem{lemma}{Lemma}

\begin{document}
\title{Gruss-type inequality by mean of a fractional integral}
\author{J. Vanterler da C. Sousa$^1$}
\address{$^1$ Department of Applied Mathematics, Institute of Mathematics,
 Statistics and Scientific Computation, University of Campinas --
UNICAMP, rua S\'ergio Buarque de Holanda 651,
13083--859, Campinas SP, Brazil\newline
e-mail: {\itshape \texttt{ra160908@ime.unicamp.br, ra142310@ime.unicamp.br, capelas@ime.unicamp.br }}}

\author{D. S. Oliveira$^1$}

\author{E. Capelas de Oliveira$^1$}

\begin{abstract}In this paper, using a fractional integral as proposed by Katugampola we establish a generalization of integral inequalities of Gruss-type. We prove two theorems associated with these inequalities and then immediately we enunciate and prove others inequalities associated with these fractional operator.
\end{abstract}
\maketitle
\section{Introduction}

In 1935, Gruss proved the following integral inequality \cite{Gruss}
\begin{eqnarray}
&&\left\vert \frac{1}{b-a}\int_{a}^{b}f\left( x\right) g\left( x\right)
dx-\left( \frac{1}{b-a}\int_{a}^{b}f\left( x\right) dx\right) \left( \frac{1%
}{b-a}\int_{a}^{b}g\left( x\right) dx\right) \right\vert\nonumber\\
&&\leq \frac{\left(
M-m\right) \left( P-p\right) }{4}, \label{Gruss}
\end{eqnarray}
where $f$ and $g$ are two integrable functions on $[a,b]$ satisfying the conditions 
\begin{eqnarray}
m\leq{f(x)}\leq{M}, \quad p\leq{g(x)}\leq{P}; \quad m,M,p,P\in\mathbb{R}, \quad x\in[a,b].
\end{eqnarray}
The Gruss type inequality has some important applications. We mention: difference equations, integral arithmetic mean and $h$-integral arithmetic mean \cite{MNCL,AESG}. On the other hand, we also study the Gruss type inequality in spaces with intern product, consequently 
some applications for the Mellin transform of sequences and polynomials in Hilbert spaces 
\cite{ABS}.

In this sense, there are another important inequalities using integer order integrals among
than we mention: H\"{o}lder's inequality, Jensen's inequality, Minkowski's inequality and 
reverse Minkowiski's inequality, \cite{HO,SEM,SWK,BL,HO1,BP,OJC}. For such inequalities, 
as well as for the study of functions, integrals and norms, the space of the $p-$integrable 
functions, $L^{p}(a,b)$, have a particular importance. However, in this work, we use the 
space of Lebesgue mensurable functions, which admits, as a particular case, the space 
$L^{p}(a,b)$.
 
The appearance of the fractional calculus allow several consequences, results and important 
theories in mathematics, physics, engineering, among others areas. Due to this fact, it was 
possible to define several fractional integrals, for example: Riemann-Liouville, Katugampola, 
Hadamard, Erd\'{e}lyi-Kober, Liouville and Weyl types. There are other fractional 
integrals that can be found in \cite{kilbas}. Thus, with these fractional integrals
some inequalities involving such formulations have been developed over the years, for example:
Reverse Minkowiski, Hermite-Hadarmad, Ostrowski e Fej\'{e}r inequalities, 
\cite{BL,DZ,CVP,TSB,SME,CHU,IIM,IIM1,vanterler}. We also mention that, in the literature, 
there are generalizations for the Eq.(\ref{Gruss}) using, for example, fractional integrals 
of Riemann-Liouville, Hadamard and also the $q-$fractional integral, \cite{dahmani2010new,vaijanath,Zhu2012}. 

In this paper, using a fractional integral recently introduced \cite{katugampola}, we
propose a new generalization of Eq.(\ref{Gruss}), i.e., new Gruss-type inequalities
that generalize inequalities obtained by means of Riemann-Liouville fractional
integral \cite{dahmani2010new}.

The paper is organized as follow: In section 2, we present the Katugampola's fractional 
integral as well as the convenient space for such definition and the convenient parameters 
to recover the six fractional integrals as particular cases. In section 3, our main result, 
we present inequalities of Gruss-type. In section 4, we discuss other inequalities involving 
the Katugampola's fractional integral. Conclusions and future perspectives close the paper.

\section{Preliminaries}

In this section, we define the space ${X}_{c}^{p}(a,b)$, where the Katumgampola's fractional integrals are defined. With a convenient choice of parameters, we recover six well-know fractional integrals, previously mentioned \cite{kilbas, katugampola}.
\begin{definition}
The space ${X}_{c}^{p}(a,b)$ $(c\in\mathbb{R}, 1\leq{p}\leq{\infty})$ consists of those complex-valued Lebesgue measurable functions $\varphi$ on $(a,b)$ for which ${\lVert \varphi \rVert}_{{X_{c}^{p}}}<\infty$, 
with
$${\lVert \varphi \rVert}_{{X_{c}^{p}}}=\left(\int_{a}^{b}|x^{c}\varphi(x)|^{p}\,\frac{dx}{x}\right)^{1/p}\; 
(1\leq{p}<{\infty})$$
and
$$\left\|\varphi\right\|_{X_{c}^{\infty}}=\sup \textnormal{ess}_{x\in(a,b)}\,[x^{c}|\varphi(x)|].$$
\end{definition}
In particular, when $c=1/p$, the space ${X}_{c}^{p}(a,b)$ coincides with the space $L^{p}(a,b)$.
\begin{definition} Let $\varphi\in X^{p}_{c}(a,b)$, $\alpha>0$ and $\beta,\rho,\eta,\kappa\in\mathbb{R}$. Then, the fractional integrals of a function $\varphi$, left- and right-sided, are defined by
\begin{equation}\label{B1}
^{\rho }\mathcal{I}_{a+,\eta ,\kappa }^{\alpha ,\beta }\varphi\left( x\right) :=\frac{\rho
^{1-\beta }x^{\kappa }}{\Gamma \left( \alpha \right) }\int_{a}^{x}\frac{%
\tau^{\rho \left( \eta +1\right) -1}}{\left( x^{\rho }-\tau^{\rho }\right)
^{1-\alpha }}\varphi\left( \tau\right) d\tau,\text{ }0\leq a<x<b\leq \infty
\end{equation}%
and
\begin{equation}\label{B2}
^{\rho }\mathcal{I}_{b-,\eta, \kappa }^{\alpha ,\beta }\varphi\left( x\right) :=\frac{\rho
^{1-\beta }x^{\rho \eta }}{\Gamma \left( \alpha \right) }\int_{x}^{b}\frac{%
\tau^{\kappa +\rho -1}}{\left( \tau^{\rho }-x^{\rho }\right) ^{1-\alpha }}\varphi\left(
\tau\right) d\tau,\text{ \ }0\leq a<x<b\leq \infty 
\end{equation}
respectively, if the integrals exist. 
\end{definition}
As previously mentioned, with a convenient choice of parameters, the fractional integral given
by \textnormal{Eq.(\ref{B1})} admits, as particular cases, six well-known fractional integrals,
namely:
\begin{enumerate}
\item \rm \text{If $\kappa=0$, $\eta=0$ and taking the limit $\rho\rightarrow 1$}, in 
\textnormal{Eq.(\ref{B1})}, we obtain the Riemann-Liouville fractional integral, \cite[p.~69]{kilbas}.
\item If $\beta=\alpha$, $\kappa=0$, $\eta=0$, taking the limit $\rho\rightarrow 0^{+}$ and using 
the $\ell$'Hospital's rule, in \textnormal{Eq.(\ref{B1})}, we obtain the Hadamard fractional integral, 
\cite[p.~110]{kilbas}.

\item When $\beta=0$ and $\kappa=-\rho(\alpha+\eta)$, in \textnormal{Eq.(\ref{B1})}, we obtain the Erd\'elyi-Kober fractional integral, \cite[p.~105]{kilbas}.

\item Also, for $\beta=\alpha$, $\kappa=0$ and $\eta=0$, in \textnormal{Eq.(\ref{B1})}, we recover the Katugampola
fractional integral, \cite{katugampola1}.

\item With the choice $\kappa=0$, $\eta=0$, $a=-\infty$ and taking the limit $\rho\rightarrow 1$, in
\textnormal{Eq.(\ref{B1})}, we have the Weyl fractional integral, \cite[p.~50]{rubeco}.

\item If $\kappa=0$, $\eta=0$, $a=0$ and taking the limit $\rho\rightarrow 1$, in \textnormal{Eq.(\ref{B1})}, 
we obtain the Liouville fractional integrals, \cite[p.~79]{kilbas}. 
\end{enumerate}

Note that, we present and discuss our new results associated with the fractional integral 
using the left-sided operator, only. Moreover, we admit $a=0$ in Eq.(\ref{B1}), in order 
to obtain
\begin{equation}\label{B3}
^{\rho }\mathcal{I}_{\eta ,\kappa }^{\alpha ,\beta }\varphi\left( x\right) =\frac{\rho
^{1-\beta }x^{\kappa }}{\Gamma \left( \alpha \right) }\int_{0}^{x}\frac{%
\tau^{\rho \left( \eta +1\right) -1}}{\left( x^{\rho }-\tau^{\rho }\right)
^{1-\alpha }}\varphi\left( \tau\right) d\tau.
\end{equation}

\section{Main results}

We enunciate and prove the following lemma in order to use it in the first theorem which generalizes 
the inequalities of Gruss-type, \textnormal{Eq.(\ref{Gruss})}. First of all, let $\alpha>0$, $x>0$ 
and $\beta,\rho,\eta,\kappa\in\mathbb{R}$. We define the following function
\begin{equation}\label{PCM}
\Lambda _{x,\kappa }^{\rho,\beta}\left( \alpha ,\eta \right) =\frac{\Gamma \left(\eta +1\right) }{\Gamma \left( \eta +\alpha+1\right) }\rho ^{-\beta}x^{\kappa +\rho \left( \eta +\alpha \right) },
\end{equation}
to simplify the development and notation.

\begin{lemma}\label{auxiliar-lemma1} Let $m,M,\beta ,\kappa \in \mathbb{R}$ and $u$ be an
integrable function on $[0,\infty )$. Then for all $x>0$, $\alpha >0$, $\rho
>0$ and $\eta \geq {0}$, we have 
\begin{eqnarray}
&&\displaystyle\Lambda _{x,\kappa }^{\rho,\beta}\left( \alpha ,\eta \right) \,{%
^{\rho }I_{\eta ,\kappa }^{\alpha ,\beta }}u^{2}(x)-({^{\rho }I_{\eta
,\kappa }^{\alpha ,\beta }}u(x))^{2}  \nonumber \\
&&\displaystyle=\left( M\Lambda _{x,\kappa }^{\rho,\beta }\left( \alpha ,\eta
\right) -{^{\rho }I_{\eta ,\kappa }^{\alpha ,\beta }}u(x)\right) \times
\left( {^{\rho }I_{\eta ,\kappa }^{\alpha ,\beta }}u(x)-m\Lambda _{x,\kappa
}^{\rho,\beta }\left( \alpha ,\eta \right) \right)   \nonumber \\
&&\displaystyle-\Lambda _{x,\kappa }^{\rho,\beta}\left( \alpha ,\eta \right) {%
^{\rho }I_{\eta ,\kappa }^{\alpha ,\beta }}(M-u(x))(u(x)-m),  \label{lema1}
\end{eqnarray}%
with $\Lambda _{x,\kappa }^{\rho,\beta}\left( \alpha ,\eta \right)$ given by {\rm Eq.(\ref{PCM})}.

\begin{proof}
Let $m,M\in \mathbb{R}$ and $u$ be an integrable function on $[0,\infty )$.
For all $\tau ,\xi \in \lbrack 0,\infty )$, we have 
\begin{eqnarray}
&&(M-u(\xi ))(u(\tau )-m)+(M-u(\tau ))(u(\xi )-m)-(M-u(\tau ))(u(\tau )-m) 
\nonumber \\
&&-(M-u(\xi ))(u(\xi )-m)=u^{2}(\tau )+u^{2}(\xi )-2u(\tau )u(\xi ).
\label{eq.auxiliar}
\end{eqnarray}%

Multiplying both sides of \textnormal{Eq.(\ref{eq.auxiliar})} by $%
\displaystyle\frac{\rho ^{1-\beta }\,x^{\kappa }}{\Gamma (\alpha )}\frac{{%
\tau ^{\rho (\eta +1)-1}}}{(x^{\rho }-\tau ^{\rho })^{1-\alpha }}$, where $%
\tau \in (0,x)$, $x>0$, and integrating with respect to the variable $%
\tau $, from 0 to $x$, we obtain 
\begin{eqnarray}
&&\displaystyle(M-u(\xi ))\left( {^{\rho }I_{\eta ,\kappa }^{\alpha ,\beta }}%
u(x)-m\Lambda _{x,\kappa }^{\rho,\beta}\left( \alpha ,\eta \right) \right)
+(u(\xi )-m)  \nonumber \\
&&\displaystyle\times \left( M\Lambda _{x,\kappa }^{\rho,\beta }\left( \alpha
,\eta \right) -{^{\rho }I_{\eta ,\kappa }^{\alpha ,\beta }}u(x)\right)
-\left( {^{\rho }I_{\eta ,\kappa }^{\alpha ,\beta }}(M-u(x))(u(x)-m)\right) 
\label{eq.lema1} \\
&&-\displaystyle(M-u(\xi ))(u(\xi )-m)\Lambda _{x,\kappa }^{\rho,\beta }\left(
\alpha ,\eta \right) ={^{\rho }I_{\eta ,\kappa }^{\alpha ,\beta }}u^{2}(x) 
\nonumber \\
&&\displaystyle+u^{2}(\xi )\Lambda _{x,\kappa }^{\rho,\beta }\left( \alpha ,\eta
\right) -2u(\xi ){^{\rho }I_{\eta ,\kappa }^{\alpha ,\beta }}u(x).  \nonumber
\end{eqnarray}%
Also, multiplying \textnormal{Eq.(\ref{eq.lema1})} by $\displaystyle\frac{%
\rho ^{1-\beta }\,x^{\kappa }}{\Gamma (\alpha )}\frac{{\xi ^{\rho (\eta
+1)-1}}}{(x^{\rho }-\xi ^{\rho })^{1-\alpha }}$, where $\xi \in (0,x)$ and $%
x>0$, and integrating with respect to the variable $\xi $, from 0 to $x$,
we get 
\begin{eqnarray}
&&\displaystyle\left( {^{\rho }I_{\eta ,\kappa }^{\alpha ,\beta }}%
u(x)-m\Lambda _{x,\kappa }^{\rho,\beta}\left( \alpha ,\eta \right) \right) \left(
M\Lambda _{x,\kappa }^{\rho,\beta}\left( \alpha ,\eta \right) -{^{\rho }I_{\eta
,\kappa }^{\alpha ,\beta }}u(x)\right)   \nonumber \\
&&\displaystyle+\left( M\Lambda _{x,\kappa }^{\rho,\beta }\left( \alpha ,\eta
\right) -{^{\rho }I_{\eta ,\kappa }^{\alpha ,\beta }}u(x)\right) \left( {%
^{\rho }I_{\eta ,\kappa }^{\alpha ,\beta }}u(x)-m\Lambda _{x,\kappa }^{\rho
,\beta}\left( \alpha ,\eta \right) \right)   \nonumber \\
&&\displaystyle-\Lambda _{x,\kappa }^{\rho,\beta}\left( \alpha ,\eta \right) \,{%
^{\rho }I_{\eta ,\kappa }^{\alpha ,\beta }}(M-u(x))(u(x)-m)
\label{eq.2-lema1} \\
&&\displaystyle-\Lambda _{x,\kappa }^{\rho,\beta}\left( \alpha ,\eta \right) \,{%
^{\rho }I_{\eta ,\kappa }^{\alpha ,\beta }}(M-u(x))(u(x)-m)  \nonumber \\
&&\displaystyle=\Lambda _{x,\kappa }^{\rho,\beta}\left( \alpha ,\eta \right) \,{%
^{\rho }I_{\eta ,\kappa }^{\alpha ,\beta }}u^{2}(x)+\Lambda _{x,\kappa
}^{\rho,\beta }\left( \alpha ,\eta \right) {^{\rho }I_{\eta ,\kappa }^{\alpha
,\beta }}u^{2}(x)  \nonumber \\
&&\displaystyle-2{^{\rho }I_{\eta ,\kappa }^{\alpha ,\beta }}u(x)\,{^{\rho
}I_{\eta ,\kappa }^{\alpha ,\beta }}u(x),  \nonumber
\end{eqnarray}
where we have introduced $\Lambda _{x,\kappa }^{\rho,\beta}\left( \alpha ,\eta \right)$ in {\rm Eq.(\ref{PCM})}.

Rearranging \textnormal{Eq.(\ref{eq.2-lema1})}, we immediately get, 
\textnormal{Eq.(\ref{lema1})}. 
\end{proof}
\end{lemma}

Note that, when $\eta=0,$ $\kappa=0$ and $\rho\rightarrow1$ in \textnormal{Eq.(\ref{lema1})}, we have
\begin{eqnarray*}
&&\left( M\frac{x^{\alpha }}{\Gamma (\alpha +1)}-I^{\alpha }u(x)\right)
\left( I^{\alpha }u(x)-m\frac{x^{\alpha }}{\Gamma (\alpha +1)}\right)  \\
&&-\left( \frac{x^{\alpha }}{\Gamma (\alpha +1)}-I^{\alpha
}(M-u(x))(u(x)-m)\right)  \\
&=&\frac{x^{\alpha }}{\Gamma (\alpha +1)}-I^{\alpha }u^{2}(x)-(I^{\alpha
}u(x))^{2},
\end{eqnarray*}
where
\begin{equation*}
\underset{\rho \rightarrow 1}{\lim }\Lambda _{x,0}^{\rho,\beta }\left( \alpha ,0\right) =\Lambda _{x}\left( \alpha ,0\right) =\frac{x^{\alpha }}{\Gamma \left( \alpha +1\right) }.
\end{equation*}

This result, was obtained in \textnormal{\cite{dahmani2010new}}, considering the 
Riemann-Liouville fractional integral.

\begin{theorem}
\label{Gruss-T1} Let $f$ and $g$ be two integrable functions on $[0,\infty )$%
, satisfying the condition  
\begin{equation}\label{condicoes}
m\leq {f(x)}\leq {M},\quad p\leq {g(x)}\leq {P};\quad \mbox{where}\quad
m,M,p,P\in \mathbb{R}\quad \mbox{and}\quad x\in \lbrack 0,\infty ).
\end{equation}
Then, for all $\beta ,\kappa \in \mathbb{R}$, $x>0$, $\alpha >0$, $\rho >0$
and $\eta \geq {0}$, we have 
\begin{equation}
\biggl|\Lambda _{x,\kappa }^{\rho,\beta}\left( \alpha ,\eta \right) \,{^{\rho
}I_{\eta ,\kappa }^{\alpha ,\beta }}fg(x)-{^{\rho }I_{\eta ,\kappa }^{\alpha
,\beta }}f(x)\,{^{\rho }I_{\eta ,\kappa }^{\alpha ,\beta }}g(x)\biggl|\leq {%
\left( \Lambda _{x,\kappa }^{\rho,\beta}\left( \alpha ,\eta \right) \right) ^{2}}%
(M-m)(P-p).  \label{teorema1}
\end{equation}%

\begin{proof}
Let $f$ and $g$ be two integrable functions satisfying the condition 
\textnormal{Eq.(\ref{condicoes})}. Consider 
\begin{equation*}
H(\tau ,\xi )=(f(\tau )-f(\xi ))(g(\tau )-g(\xi ));\quad \tau ,\xi \in
(0,x)\quad \mbox{and}\quad x>0,
\end{equation*}
or, in the following form, 
\begin{equation}\label{def}
H(\tau ,\xi )=f(\tau )g(\tau )-f(\tau )g(\xi )-f(\xi )g(\tau )+f(\xi )g(\xi).
\end{equation}
Multiplying both sides of \textnormal{Eq.(\ref{def})} by $\displaystyle\frac{%
\rho ^{1-\beta }\,x^{\kappa }}{\Gamma (\alpha )}\frac{{\tau ^{\rho (\eta
+1)-1}}}{(x^{\rho }-\tau ^{\rho })^{1-\alpha }}$, where $\tau \in (0,x)$, $%
x>0$, and integrating with respect to the variable $\tau $, from 0 to $x$,
we obtain 
\begin{eqnarray}
&&\frac{\rho ^{\beta -1}x^{\kappa }}{\Gamma (\alpha )}\int_{0}^{x}\frac{{%
\tau ^{\rho (\eta +1)-1}}}{(x^{\rho }-\tau ^{\rho })^{1-\alpha }}H(\tau ,\xi
)d\tau ={^{\rho }I_{\eta ,\kappa }^{\alpha ,\beta }}fg(x)-g(\xi )\,{^{\rho
}I_{\eta ,\kappa }^{\alpha ,\beta }}f(x)  \notag \\
&&-f(\xi )\,{^{\rho }I_{\eta ,\kappa }^{\alpha ,\beta }}g(x)+\Lambda
_{x,\kappa }^{\rho,\beta }\left( \alpha ,\eta \right) f(\xi )g(\xi ).
\label{eq.teorema1}
\end{eqnarray}%
Multiplying \textnormal{Eq.(\ref{eq.teorema1})} by $\displaystyle\frac{%
\rho ^{1-\beta }\,x^{\kappa }}{\Gamma (\alpha )}\frac{{\xi ^{\rho (\eta
+1)-1}}}{(x^{\rho }-\xi ^{\rho })^{1-\alpha }}$, where $\xi \in (0,x)$, $x>0$%
, and integrating with respect to the variable $\xi $, from 0 to $x$, we
have 
\begin{eqnarray*}
&&\frac{\rho ^{2(1-\beta )}x^{2\kappa }}{\Gamma ^{2}(\alpha )}%
\int_{0}^{x}\int_{0}^{x}\frac{{\tau ^{\rho (\eta +1)-1}}}{(x^{\rho }-\tau
^{\rho })^{1-\alpha }}\frac{{\xi ^{\rho (\eta +1)-1}}}{(x^{\rho }-\xi ^{\rho
})^{1-\alpha }}H(\tau ,\xi )d\tau d\xi  \\
&=&2\left( \Lambda _{x,\kappa }^{\rho,\beta}\left( \alpha ,\eta \right) \,{^{\rho
}I_{\eta ,\kappa }^{\alpha ,\beta }}fg(x)-{^{\rho }I_{\eta ,\kappa }^{\alpha
,\beta }}f(x)\,{^{\rho }I_{\eta ,\kappa }^{\alpha ,\beta }}g(x)\right) .
\end{eqnarray*}%
Applying the Cauchy-Schwarz inequality associated with double integrals, 
\textnormal{\cite{steele}}, we can write 
\begin{eqnarray}
&&\displaystyle\biggl|\Lambda _{x,\kappa }^{\rho,\beta }\left( \alpha ,\eta
\right) {^{\rho }I_{\eta ,\kappa }^{\alpha ,\beta }}fg(x)-{^{\rho }I_{\eta
,\kappa }^{\alpha ,\beta }}f(x)\,{^{\rho }I_{\eta ,\kappa }^{\alpha ,\beta }}%
g(x)\biggl|^{2}  \notag \\
&&\displaystyle\leq \left( \Lambda _{x,\kappa }^{\rho,\beta }\left( \alpha ,\eta
\right) \,{^{\rho }I_{\eta ,\kappa }^{\alpha ,\beta }}f^{2}(x)-\left( {%
^{\rho }I_{\eta ,\kappa }^{\alpha ,\beta }}f(x)\right) ^{2}\right) 
\label{cauchy-schwarz} \\
&&\displaystyle\times \left( \Lambda _{x,\kappa }^{\rho,\beta }\left( \alpha ,\eta
\right) \,{^{\rho }I_{\eta ,\kappa }^{\alpha ,\beta }}g^{2}(x)-\left( {%
^{\rho }I_{\eta ,\kappa }^{\alpha ,\beta }}g(x)\right) ^{2}\right) .  \notag
\end{eqnarray}%
Since $(M-f(x))(f(x)-m)\geq {0}$ and $(P-g(x))(g(x)-P)\geq {0}$, we have 
\begin{equation}
\Lambda _{x,\kappa }^{\rho,\beta}\left( \alpha ,\eta \right) \,{^{\rho }I_{\eta
,\kappa }^{\alpha ,\beta }}(M-f(x))(f(x)-m)\geq {0}
\end{equation}%
and 
\begin{equation}
\Lambda _{x,\kappa }^{\rho,\beta}\left( \alpha ,\eta \right) \,{^{\rho }I_{\eta
,\kappa }^{\alpha ,\beta }}(P-g(x))(g(x)-P)\geq {0}.
\end{equation}%
Thus, 
\begin{eqnarray}
&&\displaystyle\left( \Lambda _{x,\kappa }^{\rho,\beta }\left( \alpha ,\eta
\right) \,{^{\rho }I_{\eta ,\kappa }^{\alpha ,\beta }}f^{2}(x)-\left( {%
^{\rho }I_{\eta ,\kappa }^{\alpha ,\beta }}f(x)\right) ^{2}\right) %
\displaystyle\leq \left( M\Lambda _{x,\kappa }^{\rho,\beta }\left( \alpha ,\eta
\right) -{^{\rho }I_{\eta ,\kappa }^{\alpha ,\beta }}f(x)\right)   \notag
\label{eq.f} \\
&&\displaystyle\times \left( {^{\rho }I_{\eta ,\kappa }^{\alpha ,\beta }}%
f(x)-m\Lambda _{x,\kappa }^{\rho,\beta}\left( \alpha ,\eta \right) \right) 
\end{eqnarray}
and 
\begin{eqnarray}
&&\displaystyle\left( \Lambda _{x,\kappa }^{\rho,\beta }\left( \alpha ,\eta
\right) \,{^{\rho }I_{\eta ,\kappa }^{\alpha ,\beta }}g^{2}(x)-\left( {%
^{\rho }I_{\eta ,\kappa }^{\alpha ,\beta }}g(x)\right) ^{2}\right) %
\displaystyle\leq \left( P\Lambda _{x,\kappa }^{\rho,\beta }\left( \alpha ,\eta
\right) -{^{\rho }I_{\eta ,\kappa }^{\alpha ,\beta }}g(x)\right)   \notag
\label{eq.g} \\
&&\displaystyle\times \left( {^{\rho }I_{\eta ,\kappa }^{\alpha ,\beta }}%
g(x)-p\Lambda _{x,\kappa }^{\rho,\beta}\left( \alpha ,\eta \right) \right) .
\end{eqnarray}
Combining \textnormal{Eq.(\ref{cauchy-schwarz})}, \textnormal{Eq.(\ref{eq.f})%
} and \textnormal{Eq.(\ref{eq.g})}, using \textnormal{Lema \ref%
{auxiliar-lemma1}}, we conclude that 
\begin{eqnarray}
&&\displaystyle\left( \Lambda _{x,\kappa }^{\rho,\beta }\left( \alpha ,\eta
\right) \,{^{\rho }I_{\eta ,\kappa }^{\alpha ,\beta }}fg(x)-{^{\rho }I_{\eta
,\kappa }^{\alpha ,\beta }}f(x)\,{^{\rho }I_{\eta ,\kappa }^{\alpha ,\beta }}%
g(x)\right) ^{2}  \label{combinacao} \\
&&\displaystyle\leq \left( M\Lambda _{x,\kappa }^{\rho,\beta }\left( \alpha ,\eta
\right) -{^{\rho }I_{\eta ,\kappa }^{\alpha ,\beta }}f(x)\right) \left( {%
^{\rho }I_{\eta ,\kappa }^{\alpha ,\beta }}f(x)-m\Lambda _{x,\kappa }^{\rho
,\beta}\left( \alpha ,\eta \right) \right)   \notag \\
&&\displaystyle\times \left( P\Lambda _{x,\kappa }^{\rho,\beta }\left( \alpha
,\eta \right) -{^{\rho }I_{\eta ,\kappa }^{\alpha ,\beta }}g(x)\right)
\left( {^{\rho }I_{\eta ,\kappa }^{\alpha ,\beta }}g(x)-p\Lambda _{x,\kappa
}^{\rho,\beta }\left( \alpha ,\eta \right) \right) .  \notag
\end{eqnarray}%
Further, using the inequality $4ab\leq {(a+b)^{2}}$, $a,b\in \mathbb{R}$, we
have 
\begin{eqnarray}
&&\displaystyle4\left( M\Lambda _{x,\kappa }^{\rho,\beta }\left( \alpha ,\eta
\right) -{^{\rho }I_{\eta ,\kappa }^{\alpha ,\beta }}f(x)\right) \left( {%
^{\rho }I_{\eta ,\kappa }^{\alpha ,\beta }}f(x)-m\Lambda _{x,\kappa }^{\rho
,\beta}\left( \alpha ,\eta \right) \right)   \notag \\
&&\displaystyle\leq \left( \Lambda _{x,\kappa }^{\rho,\beta }\left( \alpha ,\eta
\right) (M-m)^{2}\right) ,  \label{des.m}
\end{eqnarray}%
and 
\begin{eqnarray}
&&\displaystyle4\left( P\Lambda _{x,\kappa }^{\rho,\beta }\left( \alpha ,\eta
\right) -{^{\rho }I_{\eta ,\kappa }^{\alpha ,\beta }}g(x)\right) \left( {%
^{\rho }I_{\eta ,\kappa }^{\alpha ,\beta }}g(x)-p\Lambda _{x,\kappa }^{\rho
,\beta}\left( \alpha ,\eta \right) \right)   \notag \\
&&\displaystyle\leq \left(\Lambda _{x,\kappa }^{\rho,\beta }\left( \alpha ,\eta
\right) (P-p)^{2}\right) .  \label{des.p}
\end{eqnarray}%
Finally, from \textnormal{Eq.(\ref{combinacao})}, \textnormal{Eq.(\ref{des.m}%
)} and \textnormal{Eq.(\ref{des.p})}, we obtain \textnormal{Eq.(\ref%
{teorema1})}.
\end{proof}
\end{theorem}

Applying \textnormal{Theorem \ref{Gruss-T1}} for $\alpha=1$, $\rho\rightarrow{1}$, $\eta={0}$ 
and $\kappa=0$, we obtain the classical inequality of Gruss-type, \textnormal{Eq.(\ref{Gruss})}. 
On the other hand, for $\rho\rightarrow{1}$, $\eta={0}$ and $\kappa=0$, we recover the 
\textnormal{Theorem 3.1} in \textnormal{\cite{dahmani2010new}}.\\
To prove the next theorem, which generalizes the inequalities of Gruss-type, we need the following lemmas,
in which we have introduced the notation as in Eq.(\ref{PCM}) with $\alpha=\gamma$, i.e.,
$\Lambda _{x,\kappa }^{\rho,\beta }\left( \gamma ,\eta \right) $.

\begin{lemma}
\label{auxiliar-lemma2} Let $f$ and $g$ be two integrable funtions on $%
[0,\infty )$. Then for all $\beta ,\kappa \in \mathbb{R}$, $x>0$, $\alpha >0$%
, $\rho >0$, $\eta \geq {0}$ and $\gamma >0$, we have 
\begin{eqnarray}
&&\displaystyle\left( \Lambda _{x,\kappa }^{\rho,\beta }\left( \alpha ,\eta
\right) \,{^{\rho }I_{\eta ,\kappa }^{\gamma ,\beta }}fg(x)+\Lambda
_{x,\kappa }^{\rho,\beta }\left( \gamma ,\eta \right) \,{^{\rho }I_{\eta ,\kappa
}^{\alpha ,\beta }}fg(x)\right.   \notag \\
&&\displaystyle\left. {^{\rho }I_{\eta ,\kappa }^{\alpha ,\beta }}f(x)\,{%
^{\rho }I_{\eta ,\kappa }^{\gamma ,\beta }}g(x)-{^{\rho }I_{\eta ,\kappa
}^{\gamma ,\beta }}f(x)\,{^{\rho }I_{\eta ,\kappa }^{\alpha ,\beta }}%
g(x)\right) ^{2}  \notag \\
&&\displaystyle\leq \left( \Lambda _{x,\kappa }^{\rho,\beta }\left( \alpha ,\eta
\right) \,{^{\rho }I_{\eta ,\kappa }^{\gamma ,\beta }}f^{2}(x)+\Lambda
_{x,\kappa }^{\rho,\beta }\left( \gamma ,\eta \right) \,{^{\rho }I_{\eta ,\kappa
}^{\alpha ,\beta }}f^{2}(x)\right.   \notag \\
&&\displaystyle\left. -2\,{^{\rho }I_{\eta ,\kappa }^{\alpha ,\beta }}f(x)\,{%
^{\rho }I_{\eta ,\kappa }^{\gamma ,\beta }}f(x)\right) \left( \Lambda
_{x,\kappa }^{\rho,\beta }\left( \alpha ,\eta \right) \,{^{\rho }I_{\eta ,\kappa
}^{\gamma ,\beta }}g^{2}(x)\right.   \notag \\
&&\displaystyle\left. +\Lambda _{x,\kappa }^{\rho,\beta }\left( \gamma ,\eta
\right) \,{^{\rho }I_{\eta ,\kappa }^{\alpha ,\beta }}g^{2}(x)-2\,{^{\rho
}I_{\eta ,\kappa }^{\alpha ,\beta }}g(x)\,{^{\rho }I_{\eta ,\kappa }^{\gamma
,\beta }}g(x)\right) . \label{Eq:star}
\end{eqnarray}%

\begin{proof}
Multiplying both sides of \textnormal{Eq.(\ref{eq.teorema1})} by $%
\displaystyle\frac{\rho ^{1-\beta }\,x^{\kappa }}{\Gamma (\gamma )}\frac{{%
\xi ^{\rho (\eta +1)-1}}}{(x^{\rho }-\xi ^{\rho })^{1-\gamma }}$, where $\xi
\in (0,x)$, $x>0$, and integrating with respect to the variable $\xi $,
from 0 to $x$, we obtain 
\begin{eqnarray*}
&&\frac{\rho ^{2(1-\beta )}x^{2\kappa }}{\Gamma (\alpha )\Gamma (\gamma )}%
\int_{0}^{x}\int_{0}^{x}\frac{{\tau ^{\rho (\eta +1)-1}}}{(x^{\rho }-\tau
^{\rho })^{1-\alpha }}\frac{{\xi ^{\rho (\eta +1)-1}}}{(x^{\rho }-\xi ^{\rho
})^{1-\gamma }}H(\tau ,\xi )d\tau d\xi  \\
&=&\Lambda _{x,\kappa }^{\rho,\beta }\left( \gamma ,\eta \right) \,{^{\rho
}I_{\eta ,\kappa }^{\alpha ,\beta }}fg(x)+\Lambda _{x,\kappa }^{\rho,\beta }\left(
\alpha ,\eta \right) \,{^{\rho }I_{\eta ,\kappa }^{\gamma ,\beta }}fg(x) \\
&&-{^{\rho }I_{\eta ,\kappa }^{\alpha ,\beta }}f(x)\,{^{\rho }I_{\eta
,\kappa }^{\gamma ,\beta }}f(x)-{^{\rho }I_{\eta ,\kappa }^{\alpha ,\beta }}%
g(x)\,{^{\rho }I_{\eta ,\kappa }^{\gamma ,\beta }}g(x).
\end{eqnarray*}%
Using the Cauchy-Schwarz inequality associated with double integrals, we
have 
\begin{eqnarray*}
&&\biggl|\Lambda _{x,\kappa }^{\rho,\beta }\left( \gamma ,\eta \right) \,{^{\rho
}I_{\eta ,\kappa }^{\alpha ,\beta }}fg(x)+\Lambda _{x,\kappa }^{\rho,\beta }\left(
\alpha ,\eta \right) \,{^{\rho }I_{\eta ,\kappa }^{\gamma ,\beta }}fg(x) \\
&&-\,{^{\rho }I_{\eta ,\kappa }^{\alpha ,\beta }}f(x)\,{^{\rho }I_{\eta
,\kappa }^{\gamma ,\beta }}f(x)-{^{\rho }I_{\eta ,\kappa }^{\alpha ,\beta }}%
g(x)\,{^{\rho }I_{\eta ,\kappa }^{\gamma ,\beta }}g(x)\biggl|^{2} \\
&\leq &\left( \Lambda _{x,\kappa }^{\rho,\beta }\left( \gamma ,\eta \right) {%
^{\rho }I_{\eta ,\kappa }^{\alpha ,\beta }}f^{2}(x)+\Lambda _{x,\kappa
}^{\rho,\beta }\left( \alpha ,\eta \right) \,{^{\rho }I_{\eta ,\kappa }^{\gamma
,\beta }}f^{2}(x)\right.  \\
&&\left. -2\,{^{\rho }I_{\eta ,\kappa }^{\alpha ,\beta }}f(x)\,{^{\rho
}I_{\eta ,\kappa }^{\gamma ,\beta }}f(x)\right) \left( \Lambda _{x,\kappa
}^{\rho,\beta }\left( \gamma ,\eta \right) \,{^{\rho }I_{\eta ,\kappa }^{\alpha
,\beta }}g^{2}(x)\right.  \\
&&\left. \Lambda _{x,\kappa }^{\rho,\beta}\left( \alpha ,\eta \right) \,{^{\rho
}I_{\eta ,\kappa }^{\gamma ,\beta }}g^{2}(x)-2\,{^{\rho }I_{\eta ,\kappa
}^{\alpha ,\beta }}g(x)\,{^{\rho }I_{\eta ,\kappa }^{\gamma ,\beta }}%
g(x)\right),
\end{eqnarray*}
which is \textnormal{Eq.(\ref{Eq:star})}.
\end{proof}
\end{lemma}
\begin{lemma}
\label{auxiliar-lemma3} Let $u$ be an integrable function on $[0,\infty )$
satisfying \textnormal{Eq.(\ref{condicoes})} on $[0,\infty )$. Then for all $%
\beta ,\kappa \in \mathbb{R}$, $x>0$, $\alpha >0$, $\rho >0$, $\eta \geq {0}$
and $\gamma >0$, we have 
\begin{eqnarray}\label{lema3}
&&\Lambda _{x,\kappa }^{\rho,\beta}\left( \alpha ,\eta \right) \,{^{\rho }I_{\eta
,\kappa }^{\gamma ,\beta }}u^{2}(x)+\Lambda _{x,\kappa }^{\rho,\beta }\left(
\gamma ,\eta \right) \,{^{\rho }I_{\eta ,\kappa }^{\alpha ,\beta }}u^{2}(x) 
\notag   \\
&&-2\,{^{\rho }I_{\eta ,\kappa }^{\alpha ,\beta }}u(x)\,{^{\rho }I_{\eta
,\kappa }^{\gamma ,\beta }}u(x)=\left( M\Lambda _{x,\kappa }^{\rho,\beta }\left(
\alpha ,\eta \right) -{^{\rho }I_{\eta ,\kappa }^{\alpha ,\beta }}%
u(x)\right)   \notag \\
&&\times \left( {^{\rho }I_{\eta ,\kappa }^{\gamma ,\beta }}u(x)-m\Lambda
_{x,\kappa }^{\rho,\beta }\left( \gamma ,\eta \right) \right) +\left( M\Lambda
_{x,\kappa }^{\rho,\beta }\left( \gamma ,\eta \right) -{^{\rho }I_{\eta ,\kappa
}^{\gamma ,\beta }}u(x)\right)   \notag \\
&&\times \left( {^{\rho }I_{\eta ,\kappa }^{\alpha ,\beta }}u(x)-m\Lambda
_{x,\kappa }^{\rho,\beta }\left( \alpha ,\eta \right) \right) -\Lambda _{x,\kappa
}^{\rho,\beta }\left( \alpha ,\eta \right) {^{\rho }I_{\eta ,\kappa }^{\gamma
,\beta }}(M-u(x))(u(x)-m)  \notag \\
&&-\Lambda _{x,\kappa }^{\rho,\beta }\left( \gamma ,\eta \right) \,{^{\rho
}I_{\eta ,\kappa }^{\alpha ,\beta }}(M-u(x))(u(x)-m).
\end{eqnarray}

\begin{proof}
Multiplying both sides of \textnormal{Eq.(\ref{eq.lema1})} by $\displaystyle%
\frac{\rho ^{1-\beta }\,x^{\kappa }}{\Gamma (\gamma )}\frac{{\xi ^{\rho
(\eta +1)-1}}}{(x^{\rho }-\xi ^{\rho })^{1-\gamma }}$, where $\xi \in (0,x)$%
, $x>0$, and integrating with respect to variable $\xi $, from 0 to $x$, we
obtain 
\begin{eqnarray*}
&&\left( {^{\rho }I_{\eta ,\kappa }^{\alpha ,\beta }}u(x)-m\Lambda
_{x,\kappa }^{\rho,\beta }\left( \alpha ,\eta \right) \right) \left( \frac{\rho
^{1-\beta }x^{\kappa }}{\Gamma (\gamma )}\int_{0}^{x}\frac{{\xi ^{\rho (\eta
+1)-1}}}{(x^{\rho }-\xi ^{\rho })^{1-\gamma }}(M-u(\xi ))d\xi \right)  \\
&&+\left( M\Lambda _{x,\kappa }^{\rho,\beta}\left( \alpha ,\eta \right) -{^{\rho
}I_{\eta ,\kappa }^{\alpha ,\beta }}u(x)\right) \left( \frac{\rho ^{1-\beta
}\,x^{\kappa }}{\Gamma (\gamma )}\int_{0}^{x}\frac{{\xi ^{\rho (\eta +1)-1}}%
}{(x^{\rho }-\xi ^{\rho })^{1-\gamma }}(u(\xi )-m)d\xi \right)  \\
&&-\left( {^{\rho }I_{\eta ,\kappa }^{\alpha ,\beta }}(M-u(x))(u(x)-m)\,%
\frac{\rho ^{1-\beta }\,x^{\kappa }}{\Gamma (\gamma )}\int_{0}^{x}\frac{{\xi
^{\rho (\eta +1)-1}}}{(x^{\rho }-\xi ^{\rho })^{1-\gamma }}d\xi \right)  \\
&&-\Lambda _{x,\kappa }^{\rho,\beta}\left( \alpha ,\eta \right) \,\frac{\rho
^{1-\beta }\,x^{\kappa }}{\Gamma (\gamma )}\int_{0}^{x}\frac{{\xi ^{\rho
(\eta +1)-1}}}{(x^{\rho }-\xi ^{\rho })^{1-\gamma }}(M-u(x))(u(x)-m)d\xi  \\
&=&{^{\rho }I_{\eta ,\kappa }^{\alpha ,\beta }}u^{2}(x)\left( \frac{\rho
^{1-\beta }\,x^{\kappa }}{\Gamma (\gamma )}\int_{0}^{x}\frac{{\xi ^{\rho
(\eta +1)-1}}}{(x^{\rho }-\xi ^{\rho })^{1-\gamma }}d\xi \right)  \\
&&+\Lambda _{x,\kappa }^{\rho,\beta}\left( \alpha ,\eta \right) \left( \frac{\rho
^{1-\beta }\,x^{\kappa }}{\Gamma (\gamma )}\int_{0}^{x}\frac{{\xi ^{\rho
(\eta +1)-1}}}{(x^{\rho }-\xi ^{\rho })^{1-\gamma }}u^{2}(\xi )d\xi \right) 
\\
&&-2\,{^{\rho }I_{\eta ,\kappa }^{\alpha ,\beta }}u(x)\left( \frac{\rho
^{1-\beta }\,x^{\kappa }}{\Gamma (\gamma )}\int_{0}^{x}\frac{{\xi ^{\rho
(\eta +1)-1}}}{(x^{\rho }-\xi ^{\rho })^{1-\gamma }}u(\xi )d\xi \right) .
\end{eqnarray*}%
From this last expression, follows immediately, \textnormal{Eq.(\ref{lema3})}.
\end{proof}
\end{lemma}

Considering $\eta=0$, $\kappa=0$ and $\rho\rightarrow{1}$, in \textnormal{Lemma 2 and Lemma 3}, we obtain the results of \textnormal{Lemma 3.4 and Lemma 3.5} in \textnormal{\cite{dahmani2010new}}.

To prove the next theorem, we use Lemma \ref{auxiliar-lemma2} and Lemma \ref{auxiliar-lemma3}, previously
proved.
\begin{theorem}
\label{Gruss2} Let $f$ and $g$ be two integrable functions on $[0,\infty )$
satisfying the condition \textnormal{Eq.(\ref{condicoes})} on $[0,\infty )$.
Then for all $\beta ,\kappa \in \mathbb{R}$, $x>0$, $\alpha >0$, $\gamma >0$
and $\eta \geq {0}$, we have 
\begin{eqnarray}
&&\left( \Lambda _{x,\kappa }^{\rho,\beta}\left( \alpha ,\eta \right) \,{^{\rho
}I_{\eta ,\kappa }^{\gamma ,\beta }}fg(x)+\Lambda _{x,\kappa }^{\rho,\beta }\left(
\gamma ,\eta \right) \,{^{\rho }I_{\eta ,\kappa }^{\alpha ,\beta }}%
fg(x)\right.   \notag \\
&&\left. -\,{^{\rho }I_{\eta ,\kappa }^{\alpha ,\beta }}f(x)\,{^{\rho
}I_{\eta ,\kappa }^{\gamma ,\beta }}g(x)-{^{\rho }I_{\eta ,\kappa }^{\gamma
,\beta }}f(x)\,{^{\rho }I_{\eta ,\kappa }^{\alpha ,\beta }}g(x)\right) ^{2}
\label{teorema2} \\
&\leq &\left[ \left( M\Lambda _{x,\kappa }^{\rho,\beta }\left( \alpha ,\eta
\right) -\,{^{\rho }I_{\eta ,\kappa }^{\alpha ,\beta }}f(x)\right) \left( {%
^{\rho }I_{\eta ,\kappa }^{\gamma ,\beta }}f(x)-m\Lambda _{x,\kappa }^{\rho,\beta
}\left( \gamma ,\eta \right) \right) \right.   \notag \\
&&+\left. \left( {^{\rho }I_{\eta ,\kappa }^{\alpha ,\beta }}f(x)-m\Lambda
_{x,\kappa }^{\rho,\beta }\left( \alpha ,\eta \right) \right) \left( M\Lambda
_{x,\kappa }^{\rho,\beta }\left( \gamma ,\eta \right) -{^{\rho }I_{\eta ,\kappa
}^{\gamma ,\beta }}f(x)\right) \right]   \notag \\
&&\times \left[ \left( P\Lambda _{x,\kappa }^{\rho,\beta }\left( \alpha ,\eta
\right) -{^{\rho }I_{\eta ,\kappa }^{\alpha ,\beta }}g(x)\right) \left( {%
^{\rho }I_{\eta ,\kappa }^{\gamma ,\beta }}g(x)-p\Lambda _{x,\kappa }^{\rho,\beta
}\left( \gamma ,\eta \right) \right) \right.   \notag \\
&&+\left. \left( {^{\rho }I_{\eta ,\kappa }^{\alpha ,\beta }}g(x)-p\Lambda
_{x,\kappa }^{\rho,\beta }\left( \alpha ,\eta \right) \right) \left( P\Lambda
_{x,\kappa }^{\rho,\beta }\left( \gamma ,\eta \right) -{^{\rho }I_{\eta ,\kappa
}^{\gamma ,\beta }}g(x)\right) \right] .  \notag
\end{eqnarray}%

\begin{proof}
Since $(M-f(x))(f(x)-m)\geq {0}$, $(P-g(x))(g(x)-p)\geq {0}$, $x>0$ and $%
\rho >0$, we can write 
\begin{eqnarray}
&&-\Lambda _{x,\kappa }^{\rho,\beta}\left( \alpha ,\eta \right) \,{^{\rho
}I_{\eta ,\kappa }^{\gamma ,\beta }}(M-f(x))(f(x)-m)  \notag \\
&&-\Lambda _{x,\kappa }^{\rho,\beta }\left( \gamma ,\eta \right) \,{^{\rho
}I_{\eta ,\kappa }^{\alpha ,\beta }}(M-f(x))(f(x)-m)\leq {0}  \label{desig-m}
\end{eqnarray}%
and 
\begin{eqnarray}
&&-\Lambda _{x,\kappa }^{\rho,\beta}\left( \alpha ,\eta \right) \,{^{\rho
}I_{\eta ,\kappa }^{\gamma ,\beta }}(P-g(x))(g(x)-p)  \notag \\
&&-\Lambda _{x,\kappa }^{\rho,\beta }\left( \gamma ,\eta \right) \,{^{\rho
}I_{\eta ,\kappa }^{\alpha ,\beta }}(P-g(x))(g(x)-p)\leq {0}.
\label{desig-p}
\end{eqnarray}%
Applying \textnormal{Lemma 3} for $f$ and $g$, using \textnormal{Eq.(\ref%
{desig-m})} and \textnormal{Eq.(\ref{desig-p})}, we obtain 
\begin{eqnarray}
&&\left( \Lambda _{x,\kappa }^{\rho,\beta}\left( \alpha ,\eta \right) \,{^{\rho
}I_{\eta ,\kappa }^{\gamma ,\beta }}f^{2}(x)+\Lambda _{x,\kappa }^{\rho,\beta
}\left( \gamma ,\eta \right) \,{^{\rho }I_{\eta ,\kappa }^{\gamma ,\beta }}%
f^{2}(x)\right.   \notag \\
&&\left. -2\,{^{\rho }I_{\eta ,\kappa }^{\gamma ,\beta }}\,{^{\rho }I_{\eta
,\kappa }^{\alpha ,\beta }}f(x)\right) \leq \left( M\Lambda _{x,\kappa
}^{\rho,\beta }\left( \alpha ,\eta \right) -\,{^{\rho }I_{\eta ,\kappa }^{\alpha
,\beta }}f(x)\right)   \notag \\
&&\times \left( {^{\rho }I_{\eta ,\kappa }^{\gamma ,\beta }}f(x)-m\Lambda
_{x,\kappa }^{\rho,\beta }\left( \gamma ,\eta \right) \right) +\left( {^{\rho
}I_{\eta ,\kappa }^{\alpha ,\beta }}f(x)-m\Lambda _{x,\kappa }^{\rho,\beta }\left(
\alpha ,\eta \right) \right)   \notag \\
&&\left(M\Lambda _{x,\kappa }^{\rho,\beta }\left( \gamma ,\eta \right) -{^{\rho
}I_{\eta ,\kappa }^{\gamma ,\beta }}f(x)\right)   \label{num1}
\end{eqnarray}%
and 
\begin{eqnarray}
&&\left( \Lambda _{x,\kappa }^{\rho,\beta}\left( \alpha ,\eta \right) \,{^{\rho
}I_{\eta ,\kappa }^{\gamma ,\beta }}g^{2}(x)+\Lambda _{x,\kappa }^{\rho,\beta
}\left( \gamma ,\eta \right) \,{^{\rho }I_{\eta ,\kappa }^{\gamma ,\beta }}%
g^{2}(x)\right.   \notag \\
&&\left. -2\,{^{\rho }I_{\eta ,\kappa }^{\gamma ,\beta }}\,{^{\rho }I_{\eta
,\kappa }^{\alpha ,\beta }}g(x)\right) \leq \left( P\Lambda _{x,\kappa
}^{\rho,\beta }\left( \alpha ,\eta \right) -\,{^{\rho }I_{\eta ,\kappa }^{\alpha
,\beta }}g(x)\right)   \notag \\
&&\times \left( {^{\rho }I_{\eta ,\kappa }^{\gamma ,\beta }}g(x)-p\Lambda
_{x,\kappa }^{\rho,\beta }\left( \gamma ,\eta \right) \right) +\left( {^{\rho
}I_{\eta ,\kappa }^{\alpha ,\beta }}g(x)-p\Lambda _{x,\kappa }^{\rho,\beta }\left(
\alpha ,\eta \right) \right)   \notag \\
&&\left( P\Lambda _{x,\kappa }^{\rho,\beta }\left( \gamma ,\eta \right) -{^{\rho
}I_{\eta ,\kappa }^{\gamma ,\beta }}g(x)\right)   \label{num2}
\end{eqnarray}%
Considering the product \textnormal{Eq.(\ref{num1})} by \textnormal{Eq.(%
\ref{num2})}, using \textnormal{Lemma \ref{auxiliar-lemma2}}, follows
immediately, \textnormal{Eq.(\ref{teorema2})}.
\end{proof}
\end{theorem}
Taking $\alpha=\gamma$ in \textnormal{Theorem \ref{Gruss2}}, we obtain \textnormal{Theorem \ref{Gruss-T1}}. 
On the other hand, considering $\eta=0$, $\kappa=0$ and $\rho\rightarrow{1}$, 
in \textnormal{Theorem 2}, we obtain \textnormal{Theorem 3.3} in \textnormal{\cite{dahmani2010new}}. 
Considering \textnormal{Theorem \ref{Gruss2}} with $\alpha=\gamma=1$, $\eta=0$, $\kappa=0$ and 
$\rho\rightarrow{1}$, we recover \textnormal{Eq.(\ref{Gruss})}.\\
We mention that, the results obtained in \textnormal{Lemmas \ref{auxiliar-lemma1}, 
\ref{auxiliar-lemma2}, \ref{auxiliar-lemma3}} and \textnormal{Theorems \ref{Gruss-T1} 
and \ref{Gruss2}}, can be proved, considering the integral in \textnormal{Eq.(\ref{B3})}, 
from $a=1$ to $x$, in order to obtain, as a particular cases, the generalized inequalities
of Gruss-type, discussed in \textnormal{\cite{vaijanath}}, with the Hadamard fractional integral.
\section{Other fractional integral inequalities} 
In this section, we present some integral inequalities involving the Katugampola's fractional
operator. The results obtained were adapted from the paper Chinchane \& Pachpatte \textnormal{\cite{vaijanath}}, in which makes a brief approach with respect inequalities of Gruss-type, but in Hadamard sense.

A priori, it should be emphasize that, although the following results are adapted for the Katugampola's
operator, the left-sided, Eq.(\ref{B3}), and that with convenient condition on the parameters of the
fractional operator it is possible to obtain the Hadamard's operator. In this sense, the results presented 
here are, in fact, true for the Hadamard's operator when we admit $a=1$ in Eq.(\ref{B1}).

In this way, we discuss some theorems involving fractional integrals inequalities.
\begin{theorem} 
Let $\alpha>0$, $\beta,\rho,\eta,\kappa\in\mathbb{R}$ and $f,g\in X^{p}_{c}(0,x)$ be two positive functions defined on $[0,\infty)$, $x>0$ and $p,q>1$ satisfying $\frac{1}{p}+\frac{1}{q}=1$. Then the following inequalities hold:

\begin{enumerate}
\item[\textnormal 1.] $\displaystyle\frac{^{\rho }\mathcal{I}_{\eta ,\kappa }^{\alpha ,\beta }f^{p}\left( x\right) }{p}
+\frac{^{\rho }\mathcal{I}_{\eta ,\kappa }^{\alpha ,\beta }g^{q}\left( x\right) }{q}\geq \frac{\Gamma 
\left( \eta +\alpha +1\right) \rho ^{\beta }}{\Gamma\left( \eta +1\right) x^{\rho \left( \eta +\alpha \right) 
+\kappa }}\left(^{\rho }\mathcal{I}_{\eta ,\kappa }^{\alpha ,\beta }f\left( x\right) ^{\rho }
\mathcal{I}_{\eta,\kappa }^{\alpha ,\beta }g\left( x\right) \right).$

\item[\textnormal 2.] $\displaystyle\frac{^{\rho }\mathcal{I}_{\eta ,\kappa }^{\alpha ,\beta }f^{p}\left( x\right)^{\rho}
\mathcal{I}_{\eta ,\kappa }^{\alpha ,\beta }g^{p}\left( x\right) }{p}+\frac{^{\rho }
\mathcal{I}_{\eta ,\kappa }^{\alpha ,\beta }f^{q}\left( x\right) ^{\rho }I_{\eta ,\kappa }^{\alpha ,\beta }
g^{q}\left( x\right) }{q}\geq \left( ^{\rho }\mathcal{I}_{\eta ,\kappa }^{\alpha ,\beta }
f\left( x\right) g\left( x\right) \right) ^{2}.$

\item[\textnormal 3.] $\displaystyle\frac{^{\rho }\mathcal{I}_{\eta ,\kappa }^{\alpha ,\beta }
f^{p}\left( x\right) ^{\rho }\mathcal{I}_{\eta ,\kappa }^{\alpha ,\beta }g^{q}\left( x\right) }{p}+
\frac{^{\rho }\mathcal{I}_{\eta ,\kappa }^{\alpha ,\beta }f^{q}\left( x\right) ^{\rho }
\mathcal{I}_{\eta ,\kappa }^{\alpha ,\beta }g^{p}\left( x\right) }{q}\geq \left( ^{\rho }
\mathcal{I}_{\eta ,\kappa }^{\alpha ,\beta }\left( fg\right) ^{p-1}\left( x\right) \right) 
\left( ^{\rho }\mathcal{I}_{\eta ,\kappa }^{\alpha ,\beta }\left( fg\right) ^{q-1}\left( x\right) 
\right).$

\item[\textnormal 4.] $\displaystyle^{\rho }\mathcal{I}_{\eta ,\kappa }^{\alpha ,\beta }f^{p}\left( x\right) ^{\rho }
\mathcal{I}_{\eta ,\kappa }^{\alpha ,\beta }g^{q}\left( x\right) \geq \left( ^{\rho }
\mathcal{I}_{\eta ,\kappa }^{\alpha ,\beta }f\left( x\right) g\left( x\right) \right)
\left( ^{\rho }\mathcal{I}_{\eta ,\kappa }^{\alpha ,\beta }f^{p-1}\left( x\right) g^{q-1}
\left( x\right) \right).$
\end{enumerate}
\end{theorem}

\begin{proof} 1. Considering Young inequality \cite{KREZ},
\begin{equation}\label{A1}
ab\leq \frac{a^{p}}{p}+\frac{b^{q}}{q},\text{ }\forall a,b\geq 0,\text{ }p,q>1, \text{ } \frac{1}{p}+\frac{1}{q}=1,
\end{equation}
and putting $a=f(t)$ and $b=f(s)$, $s>0$, in inequality {\rm Eq.\rm(\ref{A1})}, we have
\begin{equation}\label{A2}
\frac{f^{p}\left( t\right) }{p}+\frac{g^{q}\left( s\right) }{q}\geq f\left( t\right) g\left( s\right) 
,\text{ }\forall f\left( t\right) g\left( s\right) \geq 0.
\end{equation}

Multiplying by $\displaystyle\frac{\rho ^{1-\beta }x^{\kappa }t^{\rho \left( \eta +1\right) -1}}
{\Gamma\left( \alpha \right) \left( x^{\rho }-t^{\rho }\right) ^{1-\alpha }}$ both sides of {\rm Eq.\rm(\ref{A2})}, and integrating with respect to the variable $t$ on $\in (0,x)$, $x>0$, we have
\begin{eqnarray}
&&\frac{\rho ^{1-\beta }x^{\kappa }}{p\Gamma \left( \alpha \right) }%
\int_{0}^{x}\frac{t^{\rho \left( \eta +1\right) -1}}{\left( x^{\rho
}-t^{\rho }\right) ^{1-\alpha }}f^{p}\left( t\right) dt+\frac{g^{q}\left(
s\right) \rho ^{1-\beta }x^{\kappa }}{q\Gamma \left( \alpha \right) }%
\int_{0}^{x}\frac{t^{\rho \left( \eta +1\right) -1}}{\left( x^{\rho
}-t^{\rho }\right) ^{1-\alpha }}dt  \notag \\
&\geq &\frac{g\left( s\right) \rho ^{1-\beta }x^{\kappa }}{p\Gamma \left(
\alpha \right) }\int_{0}^{x}\frac{t^{\rho \left( \eta +1\right) -1}}{\left(
x^{\rho }-t^{\rho }\right) ^{1-\alpha }}f\left( t\right) dt,
\end{eqnarray}
which can be rewritten as follows,
\begin{equation}\label{A3}
\frac{^{\rho }\mathcal{I}_{\eta ,\kappa }^{\alpha ,\beta }f^{p}\left( x\right) }{p}+\frac{g^{q}
\left( s\right) \rho ^{1-\beta }x^{\kappa }}{q\Gamma \left(\alpha \right) }
\int_{0}^{x}\frac{t^{\rho \left( \eta +1\right) -1}}{\left(x^{\rho }-t^{\rho }\right)^{1-\alpha }}dt
\geq g\left( s\right) ^{\rho}\mathcal{I}_{\eta ,\kappa }^{\alpha ,\beta }f\left( x\right) .  
\end{equation}
Further with the variable change $u=\frac{t^{\rho }}{x^{\rho }}$ in the integral 
$\displaystyle\int_{0}^{x}\frac{t^{\rho \left( \eta +1\right) -1}}{\left( x^{\rho }-t^{\rho }
\right) ^{1-\alpha }}dt$, we obtain
\begin{equation}\label{A4}
\int_{0}^{x}\frac{t^{\rho \left( \eta +1\right) -1}}{\left( x^{\rho
}-t^{\rho }\right) ^{1-\alpha }}dt=\int_{0}^{1}\frac{u^{\eta }x^{\rho \eta }}
{x^{\rho \left( 1-\alpha \right) }\left( 1-u\right) ^{1-\alpha }}\frac{x^{\rho }}{\rho }du=
\frac{x^{\rho\left( \eta +\alpha \right) }}{\rho }B\left( \eta +1,\alpha \right) ,
\end{equation}
where $B(a,b)$ is the Beta function.
Using the following identity $B\left( a,b\right) =\frac{\Gamma \left( a\right) \Gamma \left( b\right)}
{\Gamma \left( a+b\right) }$ in {\rm Eq.\rm(\ref{A4})} and replacing the result in {\rm Eq.\rm(\ref{A3})}, 
we have
\begin{equation}\label{A5}
\frac{^{\rho }\mathcal{I}_{\eta ,\kappa }^{\alpha ,\beta }f^{p}\left( x\right) }{p}+
\frac{g^{q}\left( s\right) x^{\rho \left( \eta +\alpha \right) +\kappa }}{q\rho ^{\beta}}
\frac{\Gamma \left( \eta +1\right) }{\Gamma \left( \eta+\alpha +1\right) }\geq g\left( s\right)^{\rho }
\mathcal{I}_{\eta ,\kappa }^{\alpha ,\beta }f\left( x\right).
\end{equation}
Multiplying by $\displaystyle\frac{\rho ^{1-\beta }s^{\rho \left( \eta +1\right) -1}}
{\Gamma \left( \alpha \right) \left( x^{\rho }-s^{\rho }\right) ^{1-\alpha }}$ both
sides of {\rm Eq.\rm(\ref{A5})} and integrating with respect to the variable $s$ on $\in(0,x)$, $x>0$, we have
\begin{eqnarray*}
&&\frac{^{\rho }\mathcal{I}_{\eta ,\kappa }^{\alpha ,\beta }f^{p}\left( x\right) }
{p\Gamma \left( \alpha \right) }\rho ^{1-\beta }x^{\kappa }\int_{0}^{x}
\frac{s^{\rho \left( \eta +1\right) -1}}{\left( x^{\rho }-s^{\rho }\right) ^{1-\alpha }}ds
+\frac{\Gamma \left( \eta +1\right) x^{\rho \left( \eta+\alpha \right) +2\kappa }\rho ^{1-\beta }}
{\Gamma \left( \alpha \right)\Gamma \left( \eta +\alpha +1\right) q\rho ^{\beta }}
\int_{0}^{x}\frac{s^{\rho\left( \eta +1\right) -1}}{\left( x^{\rho }-s^{\rho }\right)^{1-\alpha }}
g^{q}\left( s\right) ds \notag \\
&\geq &\left( ^{\rho }\mathcal{I}_{\eta ,\kappa }^{\alpha ,\beta }f\left( x\right) \right) 
\frac{\rho ^{1-\beta }x^{\kappa }}{\Gamma \left( \alpha \right) }\int_{0}^{x}
\frac{s^{\rho \left( \eta +1\right) -1}}{\left( x^{\rho}-s^{\rho }\right) ^{1-\alpha}}g\left( s\right) ds,
\end{eqnarray*}
which can be rewritten as follows,
\begin{eqnarray}
&&\bigskip \frac{^{\rho }\mathcal{I}_{\eta ,\kappa }^{\alpha ,\beta
}f^{p}\left( x\right) \rho ^{1-\beta }x^{\kappa }}{p\Gamma \left( \alpha
\right) }\int_{0}^{x}\frac{s^{\rho \left( \eta +1\right) -1}}{\left( x^{\rho
}-t^{\rho }\right) ^{1-\alpha }}ds+\frac{x^{\rho \left( \eta +\alpha \right)
+\kappa }\Gamma \left( \eta +1\right) }{q\rho ^{\beta }\Gamma \left( \eta
+\alpha +1\right) }\left( ^{\rho }\mathcal{I}_{\eta ,\kappa }^{\alpha ,\beta
}g^{q}\left( x\right) \right)   \notag  \label{A6} \\
&\geq &\left( ^{\rho }\mathcal{I}_{\eta ,\kappa }^{\alpha ,\beta }f\left(
x\right) ^{\rho }\mathcal{I}_{\eta ,\kappa }^{\alpha ,\beta }g\left(
x\right) \right) .
\end{eqnarray}
Further with the variable change $u=\frac{s^{\rho }}{x^{\rho }}$ in integral
$\displaystyle\int_{0}^{x}\frac{s^{\rho \left( \eta +1\right) -1}}{\left( x^{\rho }-t^{\rho }
\right) ^{1-\alpha }}ds$, we obtain
\begin{equation}\label{A7}
\int_{0}^{x}\frac{s^{\rho \left( \eta +1\right) -1}}{\left( x^{\rho
}-t^{\rho }\right) ^{1-\alpha }}ds=\frac{x^{\rho \left( \eta +\alpha
\right) }\Gamma \left( \eta +1\right) \Gamma \left( \alpha \right) }{\rho
\Gamma \left( \eta +\alpha +1\right) }.
\end{equation}
Replacing {\rm Eq.\rm(\ref{A7})} in {\rm Eq.\rm(\ref{A6})}, we conclude that
\begin{equation*}
\bigskip \frac{^{\rho }\mathcal{I}_{\eta ,\kappa }^{\alpha ,\beta }f^{p}\left(
x\right) }{p}+\frac{^{\rho }\mathcal{I}_{\eta ,\kappa }^{\alpha ,\beta }g^{p}\left( x\right) }{q}\geq \frac{\Gamma \left( \eta +\alpha +1\right) \rho ^{\beta }}{\Gamma 
\left( \eta +1\right) x^{\rho \left( \eta +\alpha \right) +\kappa }}\left( ^{\rho }\mathcal{I}_{\eta ,\kappa }^{\alpha ,\beta }f\left( x\right) ^{\rho}\mathcal{I}_{\eta ,\kappa }^{\alpha ,\beta }g\left( x\right) \right) .
\end{equation*}

\item 2.  For the proof item {\rm (2)}, we take $a=f\left( t\right)g\left(s\right)$ 
and $b=f\left(s\right)g\left(t\right) $ and replacing in Young inequality following 
the same as in item {\rm(1)}.

\item 3. To prove item {\rm(3)}, we take $a=\frac{f\left( t\right) }{g\left( t\right) }$ 
and $b=\frac{f\left( s\right) }{g\left( s\right) }$, then replacing in Young inequality, 
in the same way as in item {\rm(1)}.
\item 4. Putting $a=\frac{f\left( s\right) }{f\left( t\right) }$ and 
$b=\frac{g\left( s\right) }{g\left( t\right) }$, $f(t), g(t)\neq 0$, 
and replacing in the {\rm Eq.\rm(\ref{A1})}, we get
\begin{equation}\label{A9}
\frac{f^{p}\left( t\right) }{p}+\frac{g^{q}\left( s\right) }{q}\geq f\left( t\right) g\left( s\right),
\text{ }\forall f\left( t\right) g\left( s\right) \geq 0.
\end{equation}
Multiplying by $\displaystyle\frac{\rho ^{1-\beta }x^{\kappa }t^{\rho \left( \eta +1\right) -1}}
{\Gamma \left( \alpha \right) \left( x^{\rho }-t^{\rho }\right) ^{1-\alpha }}$ both sides of 
{\rm Eq.\rm(\ref{A9})}, and integrating with respect to the variable $t$ on $(0,x)$, $x>0$, we have
\begin{equation}\label{A10}
f^{p}\left( s\right) \left( ^{\rho }\mathcal{I}_{\eta ,\kappa }^{\alpha ,\beta
}g^{q}\left( x\right) \right) +g^{q}\left( s\right) \left( ^{\rho }\mathcal{I}_{\eta ,\kappa }^{\alpha ,\beta }
f^{p}\left( x\right) \right) \geq f\left( s\right) g\left( s\right) \left( ^{\rho }
\mathcal{I}_{\eta ,\kappa }^{\alpha ,\beta }f^{p-1}\left( x\right) g^{q-1}\left( x\right) \right) .
\end{equation}
Again multiplying by $\displaystyle\frac{\rho ^{1-\beta }x^{\kappa }s^{\rho \left( \eta +1\right) -1}} {\Gamma \left( \alpha \right) \left( x^{\rho }-t^{\rho }\right) ^{1-\alpha }}$ both sides of {\rm Eq.\rm(\ref{A10})}, and integrating with respect to the variable $s$ on $(0,x)$, $x>0$, we get
\begin{eqnarray}\label{A11}
&&\frac{^{\rho }\mathcal{I}_{\eta ,\kappa}^{\alpha ,\beta}g^{q}\left( x\right)\rho ^{1-\beta }
x^{\kappa}}{p\Gamma \left( \alpha \right) }\int_{0}^{x}\frac{s^{\rho\left( \eta +1\right) -1}}
{\left( x^{\rho }-s^{\rho }\right) ^{1-\alpha }}f^{p}\left( s\right) ds+\frac{^{\rho }
\mathcal{I}_{\eta ,\kappa }^{\alpha ,\beta }f^{p}\left( x\right) \rho ^{1-\beta }x^{\kappa }}
{q\Gamma \left( \alpha \right) }\int_{0}^{x}\frac{s^{\rho \left( \eta +1\right) -1}}
{\left( x^{\rho }-s^{\rho }\right) ^{1-\alpha }}g^{q}\left( s\right)ds  \notag \\
&\geq &\left( ^{\rho }\mathcal{I}_{\eta ,\kappa }^{\alpha ,\beta }f^{p-1}\left(x\right) g^{q-1}
\left( x\right) \right) \frac{\rho ^{1-\beta }x^{\kappa }}{\Gamma \left( \alpha \right) }
\int_{0}^{x}\frac{s^{\rho \left( \eta +1\right) -1}}{\left( x^{\rho }-s^{\rho }\right) ^{1-\alpha }}
f\left(s\right) g\left( s\right) ds.
\end{eqnarray}
Using the identity $\frac{1}{p}+\frac{1}{q}=1$ in {\rm Eq.\rm(\ref{A11})}, we conclude that
\begin{equation*}
^{\rho }\mathcal{I}_{\eta ,\kappa }^{\alpha ,\beta }g^{q}\left( x\right) ^{\rho}
\mathcal{I}_{\eta ,\kappa }^{\alpha ,\beta }f^{p}\left( x\right) \geq \left( ^{\rho }
\mathcal{I}_{\eta ,\kappa }^{\alpha ,\beta }f\left( x\right) g\left( x\right) \right)
\left( ^{\rho }\mathcal{I}_{\eta ,\kappa }^{\alpha ,\beta }f^{p-1}\left( x\right)g^{q-1}
\left( x\right) \right),
\end{equation*}
which is the \textnormal{item (4)}.
\end{proof} 

\begin{theorem} Let $\alpha>0$, $\beta,\rho,\eta,\kappa\in\mathbb{R}$ and $f,g\in X^{p}_{c}(0,x)$ be two positive functions on $[0,\infty)$, $x>0$ and $p,q>1$ satisfying $\frac{1}{p}+\frac{1}{q}=1$. Then, follow the inequalities:
\begin{enumerate}
\item[\textnormal 1.] $\displaystyle\frac{^{\rho }\mathcal{I}_{\eta ,\kappa }^{\alpha ,\beta }f^{p}
\left( x\right) ^{\rho }\mathcal{I}_{\eta ,\kappa }^{\alpha ,\beta }g^{2}\left(  x\right)}{p}
+\frac{^{\rho }\mathcal{I}_{\eta ,\kappa }^{\alpha ,\beta }f^{2}\left( x\right) ^{\rho }
\mathcal{I}_{\eta ,\kappa }^{\alpha ,\beta }g^{q}\left( x\right) }{q}\geq \left( ^{\rho }
\mathcal{I}_{\eta ,\kappa }^{\alpha ,\beta }f\left( x\right) g\left( x\right) \right)
\left( ^{\rho }\mathcal{I}_{\eta ,\kappa }^{\alpha ,\beta }f^{\frac{2}{q}}\left(x\right) 
g^{\frac{2}{p}}\left( x\right) \right).
$

\item[\textnormal 2.] $\displaystyle\frac{^{\rho }\mathcal{I}_{\eta ,\kappa }^{\alpha ,\beta }
f^{2}\left( x\right) ^{\rho }\mathcal{I}_{\eta ,\kappa }^{\alpha ,\beta }g^{q}\left( x\right) }{p}
+\frac{^{\rho }\mathcal{I}_{\eta ,\kappa }^{\alpha ,\beta }f^{q}\left( x\right) ^{\rho }
\mathcal{I}_{\eta ,\kappa }^{\alpha ,\beta }g^{2}\left( x\right) }{q}\geq \left( ^{\rho }
\mathcal{I}_{\eta ,\kappa }^{\alpha ,\beta }f^{\frac{2}{q}}\left( x\right) g^{\frac{2}{p}}
\left( x\right) \right) \left( ^{\rho }\mathcal{I}_{\eta ,\kappa }^{\alpha ,\beta }f^{p-1}
\left( x\right) g^{q-1}\left( x\right) \right).
$

\item[\textnormal 3.] $\displaystyle^{\rho }\mathcal{I}_{\eta ,\kappa }^{\alpha ,\beta }f^{2}
\left( x\right) ^{\rho }\mathcal{I}_{\eta ,\kappa }^{\alpha ,\beta }\left( \frac{g^{p}\left( x\right) }{p}
+\frac{g^{q}\left( x\right) }{q}\right) \geq \left( ^{\rho }\mathcal{I}_{\eta ,\kappa }^{\alpha ,\beta }
f^{\frac{2}{p}}\left( x\right) g\left( x\right) \right) \left( ^{\rho }\mathcal{I}_{\eta ,\kappa }^{\alpha ,\beta }
f^{\frac{2}{q}}\left( x\right) g\left( x\right) \right).$

\end{enumerate}
\end{theorem}
\begin{proof}
\item 1. Taking $a=f\left( t\right) g^{\frac{2}{p}}\left( s\right) $ and $b=f^{\frac{2}{q}}\left( s\right) 
g\left( t\right) $ and replacing in {\rm Eq.\rm(\ref{A1})}, we have
\begin{equation}\label{A12}
\frac{f^{p}\left( t\right) g^{2}\left( s\right) }{p}+\frac{g^{q}\left(
s\right) f\left( s\right) }{q}\geq f\left( t\right) g\left( s\right) f^{\frac{2}{p}}
\left( t\right) g^{\frac{2}{q}}\left( s\right).   
\end{equation}
Multiplying by $\displaystyle\frac{\rho ^{1-\beta }x^{\kappa }t^{\rho \left( \eta +1\right) -1}}{\Gamma \left( \alpha \right) \left( x^{\rho }-t^{\rho }\right) ^{1-\alpha }}$ both sides of {\rm Eq.\rm(\ref{A12})}, and integrating with respect to the variable $t$ on $(0,x)$, 
$x>0$, we have
\begin{equation}\label{A13}
\frac{g^{2}\left( s\right) }{p}\left( ^{\rho }\mathcal{I}_{\eta ,\kappa }^{\alpha ,\beta }f^{p}
\left( x\right) \right) +\frac{f^{2}\left( s\right) }{q}\left(^{\rho }\mathcal{I}_{\eta ,\kappa }^{\alpha ,\beta }
g^{q}\left( x\right) \right) \leq g^{\frac{2}{p}}\left( s\right) f^{\frac{2}{q}}\left( s\right) \left( ^{\rho }
\mathcal{I}_{\eta ,\kappa }^{\alpha ,\beta }f\left( x\right) g\left( x\right) \right) .
\end{equation}
Again multiplying by $\displaystyle\frac{\rho ^{1-\beta }x^{\kappa }s^{\rho \left( \eta +1\right) -1}}
{\Gamma\left( \alpha \right) \left( x^{\rho }-s^{\rho }\right) ^{1-\alpha }}$ both sides of 
{\rm Eq.\rm(\ref{A13})}, and integrating with respect to the variable $s$ on $(0,x)$, $x>0$, we conclude that
\begin{equation*}
\frac{^{\rho }\mathcal{I}_{\eta ,\kappa }^{\alpha ,\beta }f^{p}\left( x\right) ^{\rho }\mathcal{I}_{\eta ,\kappa }^{\alpha ,\beta }g^{2}\left( x\right) }{p}+
\frac{^{\rho}\mathcal{I}_{\eta ,\kappa }^{\alpha ,\beta }g^{q}\left( x\right) ^{\rho }\mathcal{I}_{\eta,\kappa }^{\alpha ,\beta }f^{2}\left( x\right) }{q}\geq 
\left( ^{\rho}\mathcal{I}_{\eta ,\kappa }^{\alpha ,\beta }f\left( x\right) g\left( x\right) \right)\left( ^{\rho }\mathcal{I}_{\eta ,\kappa }^{\alpha ,\beta }
g^{\frac{2}{p}}\left(x\right) f^{\frac{2}{q}}\left( x\right) \right) .
\end{equation*}
\item 2. Taking $a=\frac{f^{\frac{2}{p}}\left( t\right) }{f\left( s\right) }$ and $b=\frac{g^{\frac{2}{q}} \left( t\right) }{g\left( s\right) }$ with $f\left( s\right) ,g\left( s\right) \neq 0$, and replacing in \textnormal{Eq.(\ref{A1})}, we have
\begin{equation*}
\frac{f^{2}\left( t\right) }{pf^{p}\left( s\right) }+\frac{g^{2}\left(
t\right) }{pg^{p}\left( s\right) }\geq \frac{f^{\frac{2}{p}}\left( t\right)
g^{\frac{2}{q}}\left( t\right) }{f\left( s\right) g\left( s\right) },
\end{equation*}
which can be rewritten as follows,
\begin{equation}\label{A14}
\frac{f^{2}\left( t\right) g^{q}\left( s\right) }{p}+\frac{g^{2}\left(
t\right) f^{p}\left( s\right) }{q}\geq f^{p-1}\left( s\right) g^{q-1}\left(
s\right) f^{\frac{2}{p}}\left( t\right) g^{\frac{2}{q}}\left( t\right).
\end{equation}
Again multiplying by $\displaystyle\frac{\rho ^{1-\beta }x^{\kappa }t^{\rho \left( \eta +1\right) -1}}
{\Gamma \left( \alpha \right) \left( x^{\rho }-t^{\rho }\right) ^{1-\alpha }}$  both sides of 
{\rm Eq.\rm(\ref{A14})}, and integrating with respect to the variable $t$ on $(0,x)$, $x>0$, we have
\begin{eqnarray*}
&&\frac{g^{q}\left( s\right) \rho ^{1-\beta }x^{\kappa }}{p\Gamma \left( \alpha \right) }\int_{0}^{x}\frac{t^{\rho \left( \eta +1\right) -1}}{\left(x^{\rho }-t^{\rho }
\right) ^{1-\alpha }}f^{2}\left( t\right) dt+\frac{f^{p}\left( s\right) \rho ^{1-\beta }x^{\kappa }}{q\Gamma \left( \alpha\right) }\int_{0}^{x}\frac{t^{\rho 
\left( \eta +1\right) -1}}{\left( x^{\rho}-t^{\rho }\right) ^{1-\alpha }}g^{2}\left( t\right) dt  \notag \\
&\geq &\frac{f^{p-1}\left( s\right) g^{q-1}\rho ^{1-\beta }x^{\kappa }}{\Gamma \left( \alpha \right) }\int_{0}^{x}\frac{t^{\rho \left( \eta
+1\right) -1}}{\left( x^{\rho }-t^{\rho }\right) ^{1-\alpha }}f^{\frac{2}{p}}\left( t\right) g^{\frac{2}{q}}\left( t\right) dt,
\end{eqnarray*}
or in the following form
\begin{equation}\label{A15}
\frac{g^{q}\left( s\right) }{p}\left( ^{\rho }\mathcal{I}_{\eta ,\kappa }^{\alpha,\beta }f^{2}\left( x\right) \right) +\frac{f^{p}\left( s\right) }{q}
\left(^{\rho }\mathcal{I}_{\eta ,\kappa }^{\alpha ,\beta }g^{2}\left( x\right) \right) \geq f^{p-1}\left( s\right) g^{q-1}\left( s\right) \left( ^{\rho }
\mathcal{I}_{\eta ,\kappa}^{\alpha ,\beta }f^{\frac{2}{p}}\left( x\right) g^{\frac{2}{q}}\left(x\right) \right).
\end{equation}
Multiplying by $\displaystyle\frac{\rho ^{1-\beta }x^{\kappa }s^{\rho \left( \eta +1\right) -1}}
{\Gamma\left( \alpha \right) \left( x^{\rho }-t^{\rho }\right) ^{1-\alpha }}$ both sides of 
{\rm Eq.\rm(\ref{A15})}, and integrating with respect to the variable $s$ on $(0,x)$, $x>0$, we conclude that
\begin{eqnarray*}
&&\frac{\left( ^{\rho }\mathcal{I}_{\eta ,\kappa }^{\alpha ,\beta
}f^{2}\left( x\right) \right) \left( ^{\rho }\mathcal{I}_{\eta ,\kappa
}^{\alpha ,\beta }g^{q}\left( x\right) \right) }{p}+\frac{\left( ^{\rho }%
\mathcal{I}_{\eta ,\kappa }^{\alpha ,\beta }g^{2}\left( x\right) \right)
\left( ^{\rho }\mathcal{I}_{\eta ,\kappa }^{\alpha ,\beta }f^{p}\left(
x\right) \right) }{q} \\
&\geq &\left( ^{\rho }\mathcal{I}_{\eta ,\kappa }^{\alpha ,\beta }f^{\frac{2%
}{p}}\left( x\right) g^{\frac{2}{q}}\left( x\right) \right) \left( ^{\rho }%
\mathcal{I}_{\eta ,\kappa }^{\alpha ,\beta }f^{p-1}\left( x\right)
g^{q-1}\left( x\right) \right) .
\end{eqnarray*}
\item 3. To prove \textnormal{item (3)}, we consider $a=\frac{f^{\frac{2}{p}}\left( t\right) }
{g\left( s\right) }$ and $b=\frac{f^{\frac{2}{q}}\left( s\right) }{g\left( t\right) }$ with $g\left( s\right), 
g\left( t\right)\neq 0$ and replace in the Young inequality, in the same way as in item {\rm(2)}. 

\end{proof}

\begin{theorem} Let $\alpha>0$, $\beta,\rho,\eta,\kappa\in\mathbb{R}$ and $f,g\in X^{p}_{c}(0,x)$ be two
positive functions on $[0,\infty)$, $x>0$ and  $p,q>1$ satisfying $\frac{1}{p}+\frac{1}{q}=1$. Let
\begin{equation}\label{D1}
m=\underset{0\leq t\leq x}{\min }\frac{f\left( t\right) }{g\left( t\right) } \qquad \text{ and } 
\qquad M=\underset{0\leq t\leq x}{\max }\frac{f\left( t\right) }{g\left(t\right) }\text{ }.
\end{equation}
Then, follow the inequalities:
\begin{enumerate}
\item[\textnormal 1.] $\displaystyle0\leq \left( ^{\rho }\mathcal{I}_{\eta ,\kappa }^{\alpha ,\beta }f^{2}
\left( x\right) ^{\rho }\mathcal{I}_{\eta ,\kappa }^{\alpha ,\beta }g^{2}
\left( x\right) \right) \leq  \frac{\left( M+m\right) ^{2}}{4Mm}\left( ^{\rho }
\mathcal{I}_{\eta ,\kappa }^{\alpha ,\beta }\left( fg\right) \left( x\right) \right) ^{2}$

\item[\textnormal 2.] $\displaystyle0\leq \sqrt{^{\rho }\mathcal{I}_{\eta ,\kappa }^{\alpha ,\beta }f^{2}
\left( x\right) ^{\rho }\mathcal{I}_{\eta ,\kappa }^{\alpha ,\beta }g^{2}\left( x\right) }-
\left( ^{\rho }\mathcal{I}_{\eta ,\kappa }^{\alpha ,\beta }\left( fg\right) \left( x\right) \right) 
\leq \frac{\left( \sqrt{M}-\sqrt{m}\right) ^{2}}{2\sqrt{Mm}}\left( ^{\rho }
\mathcal{I}_{\eta ,\kappa }^{\alpha ,\beta }\left( fg\right) \left( x\right)
\right) $

\item[\textnormal 3.] $\displaystyle0\leq ^{\rho }\mathcal{I}_{\eta ,\kappa }^{\alpha ,\beta }f^{2}\left( x\right) ^{\rho }
\mathcal{I}_{\eta ,\kappa }^{\alpha ,\beta }g^{2}\left( x\right) -
\left( ^{\rho }\mathcal{I}_{\eta ,\kappa }^{\alpha ,\beta }\left( fg\right) \left( x\right) \right) ^{2}
\leq \frac{\left( M-m\right) ^{2}}{4Mm}
\left( ^{\rho }\mathcal{I}_{\eta ,\kappa }^{\alpha ,\beta }\left( fg\right) \left( x\right) \right) ^{2}$
\end{enumerate}
\end{theorem}

\begin{proof}
\item 1. From {\rm Eq.\rm(\ref{D1})} and the inequality 
\begin{equation}
\left( \frac{f\left( t\right) }{g\left( t\right) }-m\right) \left( M-\frac{f\left( t\right) }{g\left( t\right) }\right) g^{2}\left( t\right) \geq 0,\ 0\leq t\leq x,
\end{equation}
we can write as
\begin{equation*}
\left( f\left( t\right) -mg\left( t\right) \right) \left( Mg\left( t\right)-f\left( t\right) \right) \geq 0,
\end{equation*}
obtaining the expression,
\begin{equation}\label{D2}
\left( M+m\right) f\left( t\right) g\left( t\right) \geq f^{2}\left(t\right) 
+mMg^{2}\left( t\right) .
\end{equation}
Multiplying by $\displaystyle\frac{\rho ^{1-\beta }x^{\kappa }t^{\rho \left( \eta +1\right) -1}}
{\Gamma \left( \alpha \right) \left( x^{\rho }-t^{\rho }\right) ^{1-\alpha }}$ both sides of  
{\rm Eq.\rm(\ref{D2})}, and integrating with respect to the variable $t$ on $(0,x)$, $x>0$, we have
\begin{equation}\label{D3}
\left( M+m\right) ^{\rho }\mathcal{I}_{\eta ,\kappa }^{\alpha ,\beta }f\left( x\right) 
g\left( x\right) \geq \left( ^{\rho }\mathcal{I}_{\eta ,\kappa }^{\alpha ,\beta }f^{2}
\left( x\right) \right) +mM\left( ^{\rho }\mathcal{I}_{\eta ,\kappa }^{\alpha ,\beta }
g^{2}\left( x\right) \right) .
\end{equation}
On the oher hand, it follows from $Mm>0$ and
\begin{equation}
\left( \sqrt{^{\rho }\mathcal{I}_{\eta ,\kappa }^{\alpha ,\beta }f^{2}\left( x\right) }
-\sqrt{mM\left( ^{\rho }\mathcal{I}_{\eta ,\kappa }^{\alpha ,\beta }g^{2}\left(x\right) \right) }
\right) ^{2}\geq 0,
\end{equation}
that
\begin{equation}\label{D4}
\left( ^{\rho }\mathcal{I}_{\eta ,\kappa }^{\alpha ,\beta }f^{2}\left( x\right)
\right) +mM\left( ^{\rho }\mathcal{I}_{\eta ,\kappa }^{\alpha ,\beta }g^{2}\left(x\right) \right) 
\geq 2\sqrt{\left( ^{\rho }\mathcal{I}_{\eta ,\kappa }^{\alpha ,\beta }f^{2}\left( x\right) 
\right) }\sqrt{mM\left( ^{\rho }\mathcal{I}_{\eta ,\kappa }^{\alpha ,\beta }g^{2}\left( x\right) \right)}.
\end{equation}
From {\rm Eq.\rm(\ref{D3})} and {\rm Eq.\rm(\ref{D4})}, we get
\begin{equation}\label{D5}
2\sqrt{\left( ^{\rho }\mathcal{I}_{\eta ,\kappa }^{\alpha ,\beta }f^{2}\left( x\right) \right) }
\sqrt{mM\left( ^{\rho }\mathcal{I}_{\eta ,\kappa }^{\alpha ,\beta }g^{2}\left( x\right) 
\right) }\leq \left( M+m\right) \left( ^{\rho }\mathcal{I}_{\eta ,\kappa }^{\alpha ,\beta }
f\left( x\right) g\left( x\right) \right).
\end{equation}
Thus, we conclude that
\begin{equation*}
\left( ^{\rho }\mathcal{I}_{\eta ,\kappa }^{\alpha ,\beta }f^{2}\left( x\right)
\right) \left( ^{\rho }\mathcal{I}_{\eta ,\kappa }^{\alpha ,\beta }g^{2}\left(
x\right) \right) \leq \frac{\left( M+m\right) ^{2}}{4}\left( ^{\rho }\mathcal{I}_{\eta ,\kappa }^{\alpha ,\beta }
f\left( x\right) g\left( x\right) \right) ^{2}.
\end{equation*}

\item 2. From {\rm Eq.\rm(\ref{D5})} we have
\begin{equation}\label{D6}
\sqrt{\left( ^{\rho }\mathcal{I}_{\eta ,\kappa }^{\alpha ,\beta }f^{2}\left( x\right) \right) 
\left( ^{\rho }\mathcal{I}_{\eta ,\kappa }^{\alpha ,\beta }g^{2}\left( x\right) \right) }
\leq \frac{\left( M+m\right) }{2\sqrt{mM}}\left( ^{\rho }\mathcal{I}_{\eta ,\kappa }^{\alpha ,\beta }
f\left( x\right) g\left( x\right) \right) .
\end{equation}

Subtracting ${^\rho }\mathcal{I}_{\eta ,\kappa }^{\alpha ,\beta }f\left( x\right) g\left( x\right)$ 
in both sides of {\rm Eq.\rm(\ref{D6})}, we get
\begin{eqnarray*}
\sqrt{\left( ^{\rho }\mathcal{I}_{\eta ,\kappa }^{\alpha ,\beta }f^{2}\left(
x\right) \right) \left( ^{\rho }\mathcal{I}_{\eta ,\kappa }^{\alpha ,\beta
}g^{2}\left( x\right) \right) }-\left( ^{\rho }\mathcal{I}_{\eta ,\kappa
}^{\alpha ,\beta }f\left( x\right) g\left( x\right) \right)  &\leq &\frac{%
\left( M+m\right) }{2\sqrt{mM}}\left( ^{\rho }\mathcal{I}_{\eta ,\kappa
}^{\alpha ,\beta }f\left( x\right) g\left( x\right) \right)  \\
&&-\left( ^{\rho }\mathcal{I}_{\eta ,\kappa }^{\alpha ,\beta }f\left(
x\right) g\left( x\right) \right) .
\end{eqnarray*}
Thus, we conclude that
\begin{equation*}
\sqrt{\left( ^{\rho }\mathcal{I}_{\eta ,\kappa }^{\alpha ,\beta }f^{2}\left( x\right) \right) 
\left( ^{\rho }\mathcal{I}_{\eta ,\kappa }^{\alpha ,\beta }g^{2}\left( x\right) \right) }-
\left( ^{\rho }\mathcal{I}_{\eta ,\kappa }^{\alpha ,\beta }f\left( x\right) g\left( x\right) \right) 
\leq \frac{\left( \sqrt{M}-\sqrt{m}\right) ^{2}}{2\sqrt{mM}}\left( ^{\rho }
\mathcal{I}_{\eta ,\kappa }^{\alpha ,\beta }f\left( x\right) g\left( x\right) \right) .
\end{equation*}

\item {\rm(3)} Subtracting $\left( ^{\rho }\mathcal{I}_{\eta ,\kappa }^{\alpha ,\beta }
f\left( x\right)g\left( x\right) \right) ^{2}$ of the {\rm Eq.\rm(\ref{D5})} and with 
the same procedure as in item {\rm(2)}, we conclude the proof. 
\end{proof}

\section*{Concluding remarks}

From the fractional integral unifying six existing fractional integrals, as proposed by Katugampola, it was possible to generalize inequalities of Gruss-type obtaining, as a particular case,the well-known inequality involving Riemann-Liouville fractional integral. In this sense, we also proved other inequalities using the Katugampola's fractional integral. A natural continuation of this paper, with this formulation, consists into generalize inequalities of Hermite-Hadamard and Hermite-Hadamard-Fej\'er \cite{CHU}, as well as to propose inequalities using the so-called $M$-fractional integral recently introduced \cite{edmundo}.

\bibliography{ref}
\bibliographystyle{plain}

\end{document}